%% file: main.tex
\documentclass[review]{elsarticle}

\usepackage{graphicx}
\usepackage[version=4]{mhchem}
\usepackage{longtable,tabularx}
\usepackage{csquotes}

\usepackage{geometry}
\usepackage{float}

\usepackage{url}
\usepackage{pifont}
\usepackage{empheq}
\usepackage{bm}

\usepackage{scalerel}

\usepackage{amsthm}

\usepackage{csquotes}

\usepackage{booktabs}
\usepackage{geometry}
\usepackage{float}

\usepackage{url}
\usepackage{pifont}

\usepackage{empheq}
\usepackage{bm}
\usepackage[english]{babel} 
\addto\captionsenglish{\renewcommand{\appendixname}{Appendix }}
\usepackage{amsmath,amssymb}
\usepackage{epstopdf}
\usepackage{relsize}
\usepackage{stmaryrd} 
\usepackage{subcaption}
\usepackage{multirow}
\usepackage{array,makecell}
\usepackage[squaren, Gray, cdot]{SIunits}

\usepackage{import}
\usepackage{diagbox}

\usepackage{csquotes}

\usepackage{scalerel}

\addto\captionsenglish{\renewcommand{\appendixname}{}}

\newtheorem{lemma}{Lemma}[section]
\newtheorem{rmk}{Remark}[section]

\usepackage{xcolor}
\usepackage{hyperref}

\DeclareMathOperator{\diag}{diag} 

\usepackage{float}
    \makeatletter
    \def\ps@pprintTitle{%
       \let\@oddhead\@empty
       \let\@evenhead\@empty
       \def\@oddfoot{\reset@font\hfil\thepage\hfil}
       \let\@evenfoot\@oddfoot
    }
    \makeatother

\usepackage{etoolbox}
\patchcmd{\MaketitleBox}{\footnotesize\itshape\elsaddress\par\vskip36pt}{\footnotesize\itshape\elsaddress\par\parbox[b][36pt]{\linewidth}{\vfill\hfill\textnormal{\today}\hfill\null\vfill}}{}{}%
\patchcmd{\pprintMaketitle}{\footnotesize\itshape\elsaddress\par\vskip36pt}{\footnotesize\itshape\elsaddress\par\parbox[b][36pt]{\linewidth}{\vfill\hfill\textnormal{\today}\hfill\null\vfill}}{}{}%

\begin{document}

\begin{frontmatter}


\title{Nonlinearly Stable Flux Reconstruction High-Order Methods in Split Form}
\author{Alexander {Cicchino}\corref{cor1}\fnref{CICC}}
\cortext[cor1]{Corresponding author. 
}
  \ead{alexander.cicchino@mail.mcgill.ca}
 \fntext[CICC]{Ph.D. Student Mechanical Engineering McGill University}

\author{Siva {Nadarajah}\fnref{Nadarajah}}
\fntext[Nadarajah]{Associate Professor Mechanical Engineering McGill University}

\ead{siva.nadarajah@mcgill.ca}

\author{David C. {Del Rey Fern\'andez}\fnref{DCDRF}}  
\ead{dcdelrey@gmail.com}
\fntext[DCDRF]{Research Scientist, National Institute
of Aerospace, Computational AeroSciences Branch NASA Langley Research Center (LaRC), and Center for High Performance Aerospace Computations (HiPAC)}

\date{\today}


\begin{abstract}
\input{body/abstract.tex}
\end{abstract}


\end{frontmatter}

  



\input{body/introduction.tex}

\input{body/preliminaries.tex}
\input{body/Split_Forms.tex}

\input{body/Results.tex}

\input{body/Conclusion.tex}
\input{body/Acknowledgements}
\appendix
\bibliography{bibliographie2.bib}

\end{document}

%% file: body/abstract.tex
The flux reconstruction (FR) method has gained popularity in the research community as it recovers promising high-order methods through modally filtered correction fields, such as the discontinuous Galerkin method, amongst others, on unstructured grids over complex geometries. Moreover, FR schemes, specifically energy stable FR (ESFR) schemes also known as Vincent-Castonguay-Jameson-Huynh schemes, have proven attractive as they allow for design flexibility as well as stability proofs for the linear advection problem on affine elements. Additionally, split forms have recently seen a resurgence in research activity due to their resultant nonlinear (entropy) stability proofs. This paper derives for the first time nonlinearly stable ESFR schemes in split form that enable nonlinear stability proofs for, uncollocated, modal, ESFR split forms with different volume and surface cubature nodes. The critical enabling technology is applying the splitting to the discrete stiffness operator. This naturally leads to appropriate surface and numerical fluxes, enabling both entropy stability and conservation proofs. When these schemes are recast in strong form, they differ from schemes found in the ESFR literature as the ESFR correction functions are incorporated on the volume integral. Furthermore, numerical experiments are conducted verifying that the new class of proposed ESFR split forms is nonlinearly stable in contrast to the standard split form ESFR approach. Lastly, the new ESFR split form is shown to obtain the correct orders of accuracy.

%% file: body/introduction.tex
\section{Introduction}

High-order methods such as discontinuous Galerkin (DG) and
flux reconstruction (FR), result in efficient computations via high solution accuracy and dense
computational kernels, making them an attractive approach for the exascale concurrency on current and next
generation hardware. Generally, high-order methods are known to be more efficient than low-order methods for linear hyperbolic time-dependent problems (e.g., see~\cite{kreiss1972comparison,swartz1974relative}). However, despite vigorous efforts by the research community,
their application to real world complex problems governed by nonlinear partial differential equations (PDEs) has been limited due to a lack of robustness. 

FR schemes, first presented by Huynh~\cite{huynh_flux_2007}, have proven attractive as they allow for design flexibility as well as stability proofs for the linear advection problem on affine elements. Wang and Gao~\cite{wang2009unifying} later presented an alternate approach to the FR scheme. Deemed the Lifting Collocation Penalty (LCP) approach, they considered a \enquote{correction field} applied to the surface integral; instead of reconstructing the flux across the surface of the element~\cite{wang2009unifying}. The authors merged both the FR and LCP in a common framework called Correction Procedure via Reconstruction (CPR)~\cite{wang2011adaptive,huynh2014high}. Now, FR and CPR are loosely interchangeable since Jameson \textit{et al}.~\cite{jameson_non-linear_2012} proved that the FR surface reconstruction is identical to the CPR correction field. FR and CPR were merged into a unified framework of provably linearly stable schemes,
in the form of energy stable flux reconstruction (ESFR) schemes~\cite{vincent_new_2011,wang2009unifying},
also known as Vincent-Castonguay-Jameson-Huynh (VCJH) schemes. Moreover, ESFR schemes recover other high-order schemes, such as the DG, spectral difference (SD)~\cite{liu_spectral_2006} and spectral volume, by the use of appropriate correction functions. By relating the ESFR correction functions to a DG lifting operator, Allaneau and Jameson~\cite{allaneau_connections_2011} presented one-dimensional ESFR schemes as a filtered DG scheme. This was extended and generalized by Zwanenburg and Nadarajah~\cite{zwanenburg_equivalence_2016} for up to three-dimensional elements, and allowed ESFR to be presented in an arbitrarily modal, uncollocated framework in both weak and strong forms. 

Recently, there has been a concerted research effort to extend classical entropy stability arguments to high-order methods. The original work of Tadmor~\cite{tadmor1984skew} laid a foundation enabling high-order extensions, where Tadmor~\cite{tadmor1984skew} constructed entropy conservative or stable low-order finite volume schemes. These notions were extended by LeFloch~\cite{lefloch2002fully,lefloch2000high} in the context of high-order finite difference stencils. In the last decade, these ideas were expanded to bounded domains by Fisher and co-authors~\cite{fisher2012high}, who combined the Summation-by-parts (SBP) framework with Tamdor's two-point flux functions. The core notion behind the SBP framework  is to account for discrete integration and construct operators that discretely mimic integration by parts (see the two review papers~\cite{fernandez2014review,svard2014review}).
The SBP framework has seen rapid development and extensions to various schemes including DG~\cite{gassner2013skew,chan2018discretely}, FR~\cite{ranocha2016summation,abe2018stable}, unstructured methods~\cite{fernandez2014review,crean2018entropy,Crean2019Staggered}, as well as extensions enabling entropy stability proofs
~\cite{fisher2012high,fisher2013high,crean2018entropy,chan2018discretely,FriedrichEntropy2020,parsani2015entropyWall,tadmor1984skew}.
In the context of FR, the first paper to merge a collocated DG split form and ESFR was presented by Ranocha \textit{et al}.~\cite{ranocha2016summation}, for the one-dimensional Burgers' equation, where they proved stability for the DG case. 
The methodology was further expanded for the Euler equations by Abe \textit{et al}.~\cite{abe2018stable}, where stability with ESFR in split forms was only achieved for a specific ESFR discretization; Huynh's $\text{g}_2$ lumped-Lobatto scheme~\cite{huynh_flux_2007}, which was equivalent to a collocated DG scheme on Gauss-Lobatto-Legendre nodes~\cite{de2014connections}. Neither Ranocha \textit{et al}.~\cite{ranocha2016summation} nor Abe \textit{et al}.~\cite{abe2018stable} have provided a general, nonlinearly stable ESFR scheme.

In this paper we take the first critical step towards developing provably entropy stable schemes for hyperbolic conservation laws that are broadly applicable to FR schemes, \textit{i.e.} they account for discretization design decisions such as: modal or nodal basis, uncollocated integration, different volume and surface nodes, \textit{etc}. Specifically, the development of ESFR schemes in split form that result in entropy stability proofs for Burgers' equation is considered. Although the derivation is presented in one-dimension, its extension for three-dimensional ESFR split forms is straightforward. The first main result is to perform the split form for uncollocated DG schemes within the discrete stiffness operator, rather than inverting the mass matrix and performing chain rule with respect to the differential operator. This allows the scheme to utilize the summation-by-parts property with dense norms, and results in the schemes presented by Chan~\cite{chan2018discretely}. In Section~\ref{sec:split forms}, we demonstrate that nonlinearly stable ESFR split forms are naturally derived when contructing ESFR as a filtered 
DG scheme alike Allaneau and Jameson~\cite{allaneau_connections_2011}, and Zwanenburg and Nadarajah~\cite{zwanenburg_equivalence_2016}. We propose a new class of provably nonlinearly stable, uncollocated ESFR schemes in split form by incorporating the ESFR filter on the non-conservative volume term. We also prove that, in general, no nonlinear stability claim can be made if the ESFR filter, \textit{i.e.} the influence of the ESFR correction functions, is only applied to the surface; differentiating our proposed schemes from the literature~\cite{huynh_flux_2007,huynh2014high,vincent_new_2011,allaneau_connections_2011,wang2009unifying,zwanenburg_equivalence_2016,williams_energy_2014,ranocha2016summation,abe2018stable,sheshadri2016stability,castonguay_energy_2013,jameson_non-linear_2012}. 
The proposed ESFR split form is proven to be nonlinearly stable (Section~\ref{sec:stability proof}) and conservative (Section~\ref{sec:conservation}). The theoretical proofs are numerically verified in Section~\ref{sec: num results}. The results demonstrate that a split form with the ESFR correction functions solely applied to the face is divergent, whereas incorporating the correction functions on the non-conservative volume term ensures entropy stability. Lastly, the schemes proposed by this paper are shown to achieve the correct orders of accuracy in the context of grid refinement.

%% file: body/preliminaries.tex
\section{Preliminaries}\label{sec:ESFR}

\subsection{DG Formulation}

The formulation of ESFR used in this paper follows that shown in~\cite{allaneau_connections_2011,zwanenburg_equivalence_2016,Cicchino2020NewNorm}. Consider the scalar 1D conservation law,

\begin{equation}
\begin{aligned}
    \frac{\partial}{\partial t}u( {x}^c,t) +\frac{\partial }{\partial x}f(u( {x}^c,t) )=0,t\geq 0, {x}^c\in\bm{\Omega},\\
    u( {x}^c,0)=u_0( {x}^c),
\end{aligned}
\end{equation}

\noindent where $ {f}(u( {x}^c,t) )$ stores the flux, while the $c$ superscript refers to the Cartesian coordinates. For the rest of the article, row vector notation will be used. 
The computational domain is partitioned into $M$ non-overlapping elements, denoted by $\bm{\Omega}^h$,

\begin{equation}
    \bm{\Omega} \simeq \bm{\Omega}^h \coloneqq \bigcup_{m=1}^M \bm{\Omega}_m.
\end{equation}

The global solution can then be taken as the direct sum of each approximation within each element,
\begin{equation}
    u( {x}^c,t) \simeq u^h( {x}^c,t)=\bigoplus_{m=1}^M u_m^h( {x}^c,t).
\end{equation}

On each element, we represent the solution with $N_p$ linearly independent modal or nodal polynomial basis of a maximum order of $p$; where $N_p=(p+1)^d$, $d$ is the dimension of the problem. The solution expansion on each element is,

\begin{equation}
     u_m^h( {x}^c,t)  = \sum_{i=1}^{N_p}{{\chi}_{m,i}( {x}^c)\hat{u}_{m,i}(t)}= 
     \bm{\chi}_m( {x}^c)\hat{\bm{u}}_m(t)^T,\label{eq:expansion}
\end{equation}

\noindent where 

\begin{equation}
     \bm{\chi}_m( {x}^c) \coloneqq [\chi_{m,1}( {x}^c)\text{ }\chi_{m,2}( {x}^c) \dots \chi_{m,N_p}( {x}^c)]
\end{equation}

\noindent holds the basis functions for the element. 
The elementwise residual is,

\begin{equation}
     R_m^h( {x}^c,t)=\frac{\partial}{\partial t}u_m^h( {x}^c,t) + \frac{\partial }{\partial x}f(u_m^h( {x}^c,t)).
\end{equation}

The physical coordinates are mapped through an affine mapping to the reference element \newline $ {\xi}^r\in[-1,1]$ by

\begin{equation}
      {x}_m^c( {\xi}^r)\coloneqq  {\Theta}_m( {\xi}^r)=
      \bm{\Theta}_{m}( {\xi}^r)(\hat{\bm{x}}_{m}^c)^T,
\end{equation}

\noindent where $\bm{\Theta}_{m}( {\xi}^r)\coloneqq[\Theta_{m,1}( {\xi}^r) \dots \Theta_{m,N_{t,m}}( {\xi}^r) ]$ are the mapping shape functions of the $N_{t,m}$ physical mapping control points $\hat{\bm{x}}_{m }^c \coloneqq [\hat{ {x}}_{m,1 }^c\dots \hat{ {x}}_{m, N_{t,m}}^c]$. Thus, the elementwise reference residual is,

\begin{equation}
   R_m^{h,r}(\xi^r,t)\coloneqq R_m^{h}( {\Theta}_m( {\xi}^r),t)=  \frac{\partial}{\partial t}u_m^h( {\Theta}_m( {\xi}^r),t) + \frac{1}{J_m^\Omega}\frac{\partial }{\partial \xi}f^r(u_m^h( {\Theta}_m( {\xi}^r),t)),\label{eq: residual reference}
\end{equation}
where the addition of the $r$ superscript denotes evaluation in the reference space and $J_m^\Omega$ represents the determinant of the metric Jacobian. Following a DG framework, we left multiply the residual by an orthogonal test function, and integrate over the computational domain. Choosing the test function to be the same as the basis function, applying integration 
by parts twice, and evaluating bilinear forms using cubature rules, we arrive at the following strong form: 

\begin{equation}
    \bm{M}_m\frac{d}{dt}\hat{\bm{u}}_m(t)^T+\bm{S}_{\xi}\hat{\bm{f}}_{m}^r(t)^T 
    + \sum_{f=1}^{N_f}
    \bm{\chi}(\bm{\xi}_{f}^r)^T\bm{W}_{f} \diag(\hat{\bm{n}}^r) \bm{{f}}_m^{{{C,r}}^T}    =\bm{0}^T,\label{eq:dg strong}
\end{equation}

\noindent where the derivation is not restricted to one-dimension. 
The mass and stiffness matrices are defined as,

\begin{equation}
    \begin{aligned}
      (\bm{M}_m)_{ij}\approx\int_{\bm{\Omega}_r}J_m^\Omega {\chi}_i( {\xi}^r){\chi}_j( {\xi}^r)d\bm{\Omega}_r \to \bm{M}_m= \bm{\chi}(\bm{\xi}_v^r)^T\bm{W}\bm{J}_m\bm{\chi}(\bm{\xi}_v^r),\\
      (\bm{S}_\xi)_{ij}=\int_{\bm{\Omega}_r} {\chi}_i( {\xi}^r)\frac{\partial}{\partial\xi}{\chi}_j( {\xi}^r)d\bm{\Omega}_r \rightarrow \bm{S}_\xi= \bm{\chi}(\bm{\xi}_v^r)^T\bm{W}\frac{\partial \bm{\chi}(\bm{\xi}_v^r)}{\partial \xi}.
    \end{aligned}
    \nonumber
\end{equation}

\noindent We also explicitly express the mass matrix without Jacobian dependence for later reference,

\begin{equation}
    \bm{M}= \bm{\chi}(\bm{\xi}_v^r)^T\bm{W}\bm{\chi}(\bm{\xi}_v^r).
\end{equation}

\noindent 
$\bm{\xi}_v^r$ and $\bm{\xi}_{f}^r$ refers to the reference coordinate at the $N_{vp}$ volume and $N_{fp}$ facet cubature nodes respectively; with $N_{f}$ being the number of faces on the element. 
Here $\bm{W}$ and $\bm{W}_{f}$ are diagonal operators storing the quadrature weights of integration at the volume and facet cubature nodes respectively. 
$\bm{J}_m$ is a diagonal operator storing the determinant of the metric Jacobian evaluated at the volume cubature nodes; while, $\diag(\hat{\bm{n}}^r) = \diag(\hat{{n}}^r) =\diag(\hat{n}^\xi)$ represents the diagonal matrix of the outward pointing normal on the face at the facet cubature nodes in the one-dimensional reference element. In addition, the modal coefficients of the reference flux, $\hat{\bm{f}}_{m}^r(t)$, are equivalent to the discrete $\text{L}_2$ projection of the reference flux constructed at the cubature nodes, $\bm{f}_{m}^r$; \textit{ie} $\hat{\bm{f}}_{m}^r(t)^T=\bm{\Pi}_m^p(\bm{f}_{m}^r)^T$, where $\bm{\Pi}_m^p=\bm{M}_m^{-1}\bm{\chi}(\bm{\xi}_{v}^r)^T\bm{W}\bm{J}_m$. 
Lastly, $ { \bm{f}}_m^{{{C,r}}^T}\coloneqq { \bm{f}}_m^{{*,r}^T}\!-\bm{\chi}(\bm{\xi}_{f}^r)\hat{ \bm{f}}_m^r(t)^T$ is the reference numerical flux minus the reference flux across the face $f$.

After discrete integration by parts once more, provided the cubature rule is exact for at least $2p-1$, the corresponding weak form is established,
\begin{equation}
    \bm{M}_m\frac{d}{dt}\hat{\bm{u}}_m(t)^T
    -\bm{S}_{\xi}^T\hat{\bm{f}}_{m}^r(t)^T
    + \sum_{f=1}^{N_f}\bm{\chi}(\bm{\xi}_{f}^r)^T\bm{W}_{f}\diag(\hat{\bm{n}}^r)  { \bm{f}}_m^{{{*,r}}^T} =\bm{0}^T.
\end{equation}



\subsection{Corresponding ESFR Scheme}

The ESFR scheme is derived using $p+1$ correction functions, $ {g}^{f,k}( {\xi}^r)$, specific to face $f$, facet cubature node $k$, such that,

\begin{equation}
      {g}^{f,k}( {\xi}_{f_i, k_j}^r)   \hat{  {n}}^\xi = \begin{cases}
    1 \text{ if } f_i = f,\text{ and } k_j=k\\
    0 \text{ otherwise.}
    \end{cases}\label{eq: ESFR conditions}
\end{equation}

The one-dimensional correction functions are defined by the symmetry condition \newline $g^L(\xi^r)=-g^R(-\xi^r)$ to satisfy Eq.~(\ref{eq: ESFR conditions}), and the 
fundamental assumption of ESFR~\cite[Eqs. (3.31), (3.32)]{ vincent_new_2011},

\begin{equation}
    \int_{\bm{\Omega}_r} \frac{\partial{\chi}_i({\xi}^r)}{\partial\xi} {g}^{f,k}({\xi}^r)d\bm{\Omega}_r - c\frac{\partial^p {\chi}_i( {\xi}^r)^T}{\partial\xi^p}
    \frac{\partial^{p+1} {g}^{f,k}({\xi}^r)}{\partial\xi^{p+1}}= {0},\: \forall i=1,\dots,N_p .
    \label{eq:1D ESFR assump}
\end{equation}

Here, $c$ represents the correction parameter, with values of $c_{DG}$, $c_{SD}$, $c_{HU}$, and $c_+$. Each parameter results in the scheme having different properties; where, $c_{DG}$ recovers a DG scheme, $c_{SD}$ recovers a spectral difference scheme, and $c_{HU}$ recovers Huynh's $\text{g}_2$ scheme~\cite{vincent_new_2011}. Lastly, the value of $c_+$ does not have an analytical value, but has numerically been shown to be the upper limit in a von Neumann analysis of the correction parameter before the scheme loses an order of accuracy~\cite{castonguay_phd}. Note that as $c$ increases, the maximum time step increases.

As illustrated in~\cite{zwanenburg_equivalence_2016,allaneau_connections_2011,Cicchino2020NewNorm} the corresponding ESFR strong form is,

\begin{equation}
    (\bm{M}_m+\bm{K}_m)\frac{d}{dt}\hat{\bm{u}}_m(t)^T
   +\bm{S}_{\xi} \hat{\bm{f}}_{m}^r(t)^T
	+\sum_{f=1}^{N_f}
	\bm{\chi}(\bm{\xi}_{f}^r)^T \bm{W}_{f}\diag(\hat{\bm{n}}^r) \bm{f}_m^{{C,r}^T} =\bm{0}^T,\label{eq:ESFR strong}
\end{equation}

\noindent while the corresponding weak form is,

\begin{equation}
    (\bm{M}_m+\bm{K}_m)\frac{d}{dt}\hat{\bm{u}}_m(t)^T
   -\bm{S}_{\xi}^T \hat{\bm{f}}_{m}^r(t)^T
	+\sum_{f=1}^{N_f}
	\bm{\chi}(\bm{\xi}_{f}^r)^T \bm{W}_{f}\diag(\hat{\bm{n}}^r)  \bm{{f}}_m^{{*,r}^T } =\bm{0}^T.\label{eq:weak form}
\end{equation}

The entire influence of the ESFR correction functions are stored in $\bm{K}_m$, which we define as

\begin{equation}
\begin{split}
    (\bm{K}_m)_{ij} = c \int_{ \bm{\Omega}_r} J_m^\Omega \frac{\partial ^p \chi_i( {\xi}^r)}{\partial \xi^p}\frac{\partial ^p \chi_j( {\xi}^r)}{\partial \xi^p} d \bm{\Omega_r}
    \to 
    \bm{K}_m = 
     c
    (
    \bm{D}^{p} 
    )^T 
    \bm{M}_m
     (\bm{D}^{p}),
    \end{split}\label{eq:Km}
\end{equation}
where the $p^{\text{th}}$ degree strong form differential operator~\cite{Cicchino2020NewNorm} is construed as
\begin{equation}
    \bm{D}^{p} = \Big(\bm{M}^{-1}\bm{S}_\xi\Big)^p.
\end{equation}


\begin{rmk}
Unlike in~\cite{allaneau_connections_2011,zwanenburg_equivalence_2016,ranocha2016summation} where $\bm{K}_m$ was constructed using the Legendre differential operator then transformed to the basis of the scheme, here, $\bm{K}_m$ in Eq.~(\ref{eq:Km}) is computed directly with the differential operator of the scheme; where $c$ must take the value from a normalized Legendre reference basis~\cite{Cicchino2020NewNorm}. This was proven in~\cite[Sec. 3.1]{Cicchino2020NewNorm}
\end{rmk}

%% file: body/Split_Forms.tex
\section{SBP-ESFR Split Forms for Burgers' Equation}\label{sec:split forms}

In this section we analyze a new proposed splitting that enables nonlinear stability proofs for general FR schemes, where the modal or nodal basis functions can be evaluated on uncollocated volume and surface cubature nodes. 
We will consider the Burgers' equation,

\begin{equation}
    \frac{\partial}{\partial t}u + \frac{\partial}{\partial x}(\frac{u^2}{2}) = 0.
\end{equation}

As demonstrated in~\cite{gassner2013skew,ranocha2016summation}, for a collocated DG strong form scheme, entropy and energy stablility was ensured if the differential operator is split into a linear combination of the conservative and chain rule forms. This was achieved by observing that the strong form differential operator satisfies the SBP property with respect to the metric Jacobian independent mass matrix,

\begin{equation}
\begin{split}
    \bm{M}_{\text{GLL}}\bm{D}+\bm{D}^T\bm{M}_{\text{GLL}}=\bm{B},
    \end{split}\label{eq:SBP}
\end{equation}

\noindent where 
$\bm{B}=\diag[-1,\dots, 1]$, with Gauss-Lobatto-Legendre (GLL) quadrature points or Gauss-Legendre (GL) quadrature points in one-dimension. The proposed splitting in~\cite{gassner2013skew} was,

\begin{equation}
    \begin{split}
    \frac{d}{dt}{\bm{u}}_m^T
   = \frac{1}{J_m^\Omega}\Bigg[-\alpha \bm{D}(\frac{1}{2}{\bm{U}}{\bm{u}}_m^T) 
   -(1-\alpha){\bm{U}}\bm{D}({\bm{u}}_m^T)
	-\bm{M}_{\text{GLL}} ^{-1}
\sum_{f=1}^{N_f}	\bm{\chi}(\bm{\xi}_{f}^r)^T \bm{W}_{f}\diag(\hat{\bm{n}}^r)\bm{{f}}_m^{{C,r} ^T} \Bigg],
	\end{split}\label{eq:Gass split}
\end{equation}

\noindent which was achieved using a collocated nodal Lagrange basis, with 
$\bm{U}=\diag(\bm{u}_m)$. This was further expanded for a classical ESFR scheme in~\cite{ranocha2016summation} as,

\begin{equation}
    \begin{split}
    \frac{d}{dt}{\bm{u}}_m^T
   = \frac{1}{J_m^\Omega}\Bigg[-\alpha \bm{D}(\frac{1}{2}{\bm{U}}{\bm{u}}_m^T) 
   -(1-\alpha){\bm{U}}\bm{D}({\bm{u}}_m^T)
	-(\bm{M}_{\text{GLL}}+\bm{K}) ^{-1}
	\sum_{f=1}^{N_f}	\bm{\chi}(\bm{\xi}_{f}^r)^T \bm{W}_{f}\diag(\hat{\bm{n}}^r)\bm{{f}}_m^{{C,r} ^T} \Bigg],
	\end{split}\label{eq:ranocha split}
\end{equation}

\noindent where $\bm{K}$ is the metric Jacobian independent ESFR correction operator.

To demonstrate stability, the following broken Sobolev-norm was proposed by~\cite{jameson_proof_2010} for ESFR schemes:

\begin{equation}
\begin{split}
\text{Continuous Broken Sobolev-norm: }\int_{\bm{\Omega}_r}{\Big(\bm{\chi}({\xi^r})\hat{\bm{u}}_m(t)^T \Big)^T {J}_m^\Omega \Big(\bm{\chi}({\xi^r})\frac{d}{dt}\hat{\bm{u}}_m(t)^T \Big)} d\bm{\Omega}_r+\\ \int_{\bm{\Omega}_r}{c\Big(\frac{\partial^p}{\partial \xi^p}(\bm{\chi}({\xi^r})\hat{\bm{u}}_m(t)^T)
\Big)^T {J}_m^\Omega \Big(  \frac{\partial^p}{\partial \xi^p}(\bm{\chi}({\xi^r})\frac{d}{dt}\hat{\bm{u}}_m(t)^T )
\Big)d\bm{\Omega}_r}\\
    \implies \text{Discrete Broken Sobolev-norm: }\frac{1}{2}\frac{d}{dt}\|\bm{u}\|_{M_m+K_m}^2=\hat{\bm{u}}_m(t) (\bm{M}_m+\bm{K}_m) \frac{d}{dt}\hat{\bm{u}}_m(t) ^T.
\end{split}\label{eq:init energy ESFR}
\end{equation}

 Applying the discrete broken Sobolev-norm to Eq.~(\ref{eq:ranocha split}), and using the property that $\bm{K} \bm{D} =0$ since it is the $p+1$ derivative of a $p^{\text{th}}$ order basis function, results in,

\begin{equation}
    \begin{split}
    \frac{1}{2}\frac{d}{dt}\|\bm{u}\|_{M_m+K_m}^2
   = \Bigg[-\alpha {\bm{u}}_m \bm{M}_{GLL}\bm{D}(\frac{1}{2}{\bm{U}}{\bm{u}}_m ^T) 
   -(1-\alpha){\bm{u}}_m \bm{M}_{GLL}{\bm{U}}\bm{D}({\bm{u}}_m ^T)\\
   -(1-\alpha){\bm{u}}_m \bm{K}{\bm{U}}\bm{D}({\bm{u}}_m ^T)
	-
	 \sum_{f=1}^{N_f}	{\bm{u}}_m \bm{W}_{f}\diag(\hat{\bm{n}}^r)\bm{{f}}_m^{{C,r} ^T} \Bigg].
	\end{split}\label{eq:energy ranocha}
\end{equation}

Both~\cite{gassner2013skew,ranocha2016summation} used the property that for a collocated nodal Lagrange basis, the mass matrix is a diagonal operator and therefore commutes with $\bm{U}$ in the second volume term. This then allows the use of the SBP property to relate the two volume terms to a face term. The issue that was illustrated in Ranocha \textit{et al}.~\cite{ranocha2016summation} was that for an ESFR scheme, $\bm{K} \bm{U}\bm{D}\neq \bm{0}$, nor is it in general positive semi-definite. Thus, when the split form (chain rule) is applied on the differential operator, no stability claim can be made for an ESFR scheme, unless $\bm{K}=\bm{0}$ as used by Ranocha \textit{et al}.~\cite{ranocha2016summation} or $\bm{K}$ is a diagonal operator. In Abe \textit{et al}.~\cite{abe2018stable}, they used Huynh's $\text{g}_2$ lumped-Lobatto scheme~\cite{huynh_flux_2007} which was equivalent to formulating $\bm{M}+\bm{K}$ on uncollocated Gauss-Legendre nodes, with a value of $c_{HU}$. By design, $\bm{K}_{HU}$ ($\bm{K}$ evaluated with a value of $c_{HU}$ on GL nodes) was chosen such that $\bm{M}_{GL}+\bm{K}_{HU}=\bm{M}_{GLL}$, the collocated lumped DG mass matrix on Gauss-Lobatto-Legendre nodes~\cite{huynh_flux_2007,de2014connections,vincent_new_2011}. Then, the crucial step was to evaluate the flux on GLL nodes, which results in an equivalent DG collocated GLL scheme~\cite{huynh_flux_2007,de2014connections,vincent_new_2011,abe2018stable}. In a sense, the $\text{g}_2$ lumped-Lobatto scheme operates on mixed nodes (with regards to the volume flux on GLL and correction functions on GL), since in the original FR framework~\cite{huynh_flux_2007} no concept of quadrature integration was introduced.
\begin{rmk}
For consistency, since the ESFR formulation used in this work is completely general, there is no assumption on the nodes the volume flux is integrated on, other than being exact for at least $2p-1$. Thus, in the ESFR strong and weak forms presented in Equations~(\ref{eq:ESFR strong}) and~(\ref{eq:weak form}), Huynh's $\text{g}_2$ lumped-Lobatto scheme is equivalent to using a collocated basis on GLL nodes and a value of $c_{DG}$, or using a value of $c_{HU}$ with $\bm{M}+\bm{K}$ evaluated with an integration strength of at least $2p$, 
and the volume flux on GLL nodes. Huynh's $\text{g}_2$ lumped-Lobatto scheme is not equivalent to computing both $\bm{M}+\bm{K}$ and the volume flux on the same set of nodes, with a value of $c_{HU}$. We numerically verify this in Section~\ref{sec: num results}.
\end{rmk}

\subsection{Proposed Splitting with Respect to the Stiffness Operator}

The stiffness operator satisfies a discrete integration by parts rule, \textit{i.e.} an SBP property, for quadrature rules exact for at least $2p-1$,

\begin{equation}
    \begin{split}
    \int_{\bm{\Omega}_r}{{\chi}_i( {\xi}^r)\frac{\partial}{\partial\xi}{\chi}_j( {\xi}^r)d\bm{\Omega}_r}
   +\int_{\bm{\Omega}_r}{\frac{\partial}{\partial\xi}{\chi}_i( {\xi}^r){\chi}_j( {\xi}^r)d\bm{\Omega}_r}
   =\int_{\bm{\Gamma}_r}{{\chi}_i( {\xi}^r){\chi}_j( {\xi}^r)\hat{n}^\xi d\bm{\Gamma}_r}\\
   \Leftrightarrow    \bm{S}_\xi + \bm{S}_\xi^T=\sum_{f=1}^{N_f}\bm{\chi}(\bm{\xi}_{f}^r)^T\bm{W}_{f}\diag(\hat{\bm{n}}^\xi) \bm{\chi}(\bm{\xi}_{f}^r) .
    \end{split}\label{eq: integration by parts}
\end{equation}

Since the underlying SBP property results from the fact that the stiffness operator satisfies discrete integration by parts, Eq.~(\ref{eq: integration by parts}), and observing the ESFR strong and weak forms (Eq.~(\ref{eq:ESFR strong}) and~(\ref{eq:weak form}) respectively), we propose to construct the split form based upon the stiffness operator rather than the differential operator, in contrast to previous works~\cite{gassner2013skew,ranocha2016summation,abe2018stable,chan2018discretely}. 

Returning to the continuous analog of Eq.~(\ref{eq:ESFR strong}), writing it in variational form results in,

\begin{equation}
    \begin{split}
        \int_{\bm{\Omega}_r}{ \Big( {\chi}_i( {\xi}^r)   {J}_m^\Omega \bm{\chi}( {\xi}^r) +c \frac{\partial ^p  {\chi}_i( {\xi}^r)}{\partial \xi^p }  {J}_m^\Omega \frac{\partial ^p  \bm{\chi}( {\xi}^r)}{\partial \xi^p }   \Big)\frac{d}{dt} \hat{\bm{u}}_m(t)^T d\bm{\Omega}_r} \\
        =
        -\int_{\bm{\Omega}_r}{ {\chi}_i( {\xi}^r) \frac{\partial }{\partial \xi} (\frac{1}{2}u_m^2) d\bm{\Omega}_r}
        -\int_{\bm{\Gamma}_r}{ {\chi}_i( {\xi}^r)    \hat{ {n}}^r  {{f}}_m^{C,r} d\bm{\Gamma}_r},\:\forall i=1,\dots,N_p.
    \end{split}
\end{equation}

\noindent Performing chain rule with respect to the reference coordinate gives,

\begin{equation}
    \begin{split}
        \int_{\bm{\Omega}_r}{ \Big(  {\chi}_i( {\xi}^r)  {J}_m^\Omega \bm{\chi}( {\xi}^r) +c \frac{\partial ^p   {\chi}_i({\xi}^r)}{\partial \xi^p }  {J}_m^\Omega \frac{\partial ^p  \bm{\chi}( {\xi}^r)}{\partial \xi^p }   \Big)\frac{d}{dt} \hat{\bm{u}}_m(t)^T d\bm{\Omega}_r} \\
        =
        -\alpha \int_{\bm{\Omega}_r}{  {\chi}_i( {\xi}^r) \frac{\partial }{\partial \xi} (\frac{1}{2}u_m^2)  d\bm{\Omega}_r} 
        -(1-\alpha)\int_{\bm{\Omega}_r}{  {\chi}_i( {\xi}^r)  u_m \frac{\partial }{\partial \xi}( {{u}}_m ) d\bm{\Omega}_r} \\
        -\int_{\bm{\Gamma}_r}{ {\chi}_i( {\xi}^r)   \hat{ {n}}^r  {{f}}_m^{C,r}  d\bm{\Gamma}_r},\:\forall i=1,\dots,N_p.
    \end{split}\label{eq: continuous split}
\end{equation}

Here, alike~\cite{gassner2013skew}, $\alpha\in[0,1]$. After evaluating at the appropriate cubature nodes, we invert the ESFR filter operator on the left hand side, and present it in discretized form,

\begin{equation}
\begin{split}
    \frac{d}{dt}\hat{\bm{u}}_m(t)^T
   = -\alpha (\bm{M}_m+\bm{K}_m) ^{-1}\bm{S}_{\xi}\hat{\bm{f}}_m^r(t)^T
   -(1-\alpha)(\bm{M}_m+\bm{K}_m) ^{-1}\bm{\chi}(\bm{\xi}_{v}^r)^T\bm{U}\bm{W}\frac{\partial\bm{\chi}(\bm{\xi}_{v}^r)}{\partial {\xi}}\hat{\bm{u}}_m(t)^T\\
	-(\bm{M}_m+\bm{K}_m) ^{-1}
	\sum_{f=1}^{N_f}\bm{\chi}(\bm{\xi}_{f}^r)^T \bm{W}_{f}\diag(\hat{\bm{n}}^r)\bm{{f}}_m^{{C,r}^T}.
	\end{split}
\end{equation}


Unless GLL is employed as the volume cubature nodes, the nonlinear flux interpolated to the face does not generally equal the flux on the face evaluated using the solution interpolated to the face. Therefore, we also split the flux on the face using $\bm{\chi}(\bm{\xi}_{f}^r)\hat{\bm{f}}_m^{{r}}(t)^T$ as the interpolation of the modal coefficients of the volume flux to the face, and ${ \bm{f}}_{f}^{{r}^T}=\frac{1}{2}\diag\Big(\bm{\chi}(\bm{\xi}_{f}^r)\hat{\bm{u}}_m(t)^T\Big)\Big(\bm{\chi}(\bm{\xi}_{f}^r)\hat{\bm{u}}_m(t)^T\Big)$ as the flux on face $f$ 
evaluated using the solution interpolated to the face. The final ESFR strong split form is thus,

\begin{equation}
\begin{split}
    \frac{d}{dt}\hat{\bm{u}}_m(t)^T
   = -\alpha (\bm{M}_m+\bm{K}_m) ^{-1}\bm{S}_{\xi}\hat{\bm{f}}_m^r(t)^T
   -(1-\alpha)(\bm{M}_m+\bm{K}_m) ^{-1}\bm{\chi}(\bm{\xi}_{v}^r)^T\bm{U}\bm{W}\frac{\partial\bm{\chi}(\bm{\xi}_{v}^r)}{\partial {\xi}}\hat{\bm{u}}_m(t)^T\\
	-(\bm{M}_m+\bm{K}_m) ^{-1}\sum_{f=1}^{N_f}
	\bm{\chi}(\bm{\xi}_{f}^r)^T \bm{W}_{f}\diag(\hat{\bm{n}}^r) \Big( \bm{{f}}_m^{{*,r}^T}-\alpha\bm{\chi}(\bm{\xi}_{f}^r)\hat{\bm{f}}_m^{{r}}(t)^T-(1-\alpha)\bm{{f}}_{f}^{{r}^T}\Big) .
	\end{split}\label{eq:split strong}
\end{equation}

If one were to consider the general differential operator applied to the entropy conservative two-point flux in Chan~\cite{chan2018discretely}, then the face splitting appears naturally when grouping the face lifting terms and the DG result of Eq.~(\ref{eq:split strong}) is equivalent to A.14 in~\cite{chan2018discretely}. Additionally, note that for a collocated DG scheme at Gauss-Lobatto-Legendre nodes, $\bm{\chi}(\bm{\xi}_{f}^r)\hat{\bm{f}}_m^{{r}}(t)^T=\bm{{f}}_f^{{r}^T}$, and hence no interpolation of the flux to the face was required, and thus face splitting did not appear in~\cite{gassner2013skew}.

To convert Eq.~(\ref{eq:split strong}) to the weak form, we perform discrete integration by parts, Eq.~(\ref{eq: integration by parts}), on the two volume terms. For the first volume term we directly substitute Eq.~(\ref{eq: integration by parts}). For the second volume term we first substitute $\frac{\partial\bm{\chi}(\bm{\xi}_{v}^r)}{\partial {\xi}}=\bm{\chi}(\bm{\xi}_{v}^r)\bm{M}^{-1}\bm{S}_\xi $, then introduce the discrete $\text{L}_2$ projection operator $\bm{\Pi}_m^{{p}}$ and finally Eq.~(\ref{eq: integration by parts}), 


\begin{equation}
    \begin{split}
    \frac{d}{dt}\hat{\bm{u}}_m(t)^T
   = \alpha (\bm{M}_m+\bm{K}_m) ^{-1}\bm{S}_{\xi}^T\hat{\bm{f}}_m^r(t)^T
   +(1-\alpha)(\bm{M}_m+\bm{K}_m) ^{-1}\bm{\chi}(\bm{\xi}_{v}^r)^T\bm{U}\bm{\Pi}_m^{{p}^T}\bm{S}_{\xi}^T\hat{\bm{u}}_m(t)^T\\
	-(\bm{M}_m+\bm{K}_m) ^{-1}\sum_{f=1}^{N_f}
	\bm{\chi}(\bm{\xi}_{f}^r)^T \bm{W}_{f}\diag(\hat{\bm{n}}^r)  {\bm{f}}_m^{{*,r}^T}.
	\end{split}\label{eq:split weak}
\end{equation}

\begin{rmk}
Eq.~(\ref{eq:split weak}) is not equivalent to performing integration by parts on the continuous form in Eq.~(\ref{eq: continuous split}), and then discretizing.
\end{rmk}





We will now demonstrate that the proposed ESFR split form is equivalent to the DG split form with additional volume and surface terms. For this purpose, the following lemma is necessary.

\begin{lemma}
$(\bm{M}_m +\bm{K}_m)^{-1}= \bm{M}_m^{-1}-\frac{1}{1+c(2p+1)(p!c_p)^2}\bm{M}_m^{-1}\bm{K}_m\bm{M}_m^{-1}$ for linear elements.
\end{lemma}

\begin{proof}
After factoring out the Jacobian dependence, consider transforming $(\bm{M}+\bm{K})^{-1}$ to a normalized Legendre reference basis, $\bm{\chi}_{ref}( {\xi}^r)$, by use of the transformation operator $\bm{T}=\bm{\Pi}_{ref}^P\bm{\chi}(\bm{\xi}_v^r)$, where $\bm{\Pi}_{ref}^P=\bm{M}_{ref}^{-1}\bm{\chi}_{ref}(\bm{\xi}_v^r)^T\bm{W}$ is the $\text{L}_2$ projection operator for a normalized Legendre reference basis.

\begin{equation}
    (\bm{M}+\bm{K} )^{-1} = \bm{T}^{-1}(\bm{M}_{ref} + \bm{K}_{ref})^{-1}\bm{T}^{-T}.
\end{equation}

To use the Sherman-Morrison formula, we consider $\bm{K}_{ref}=c\bm{r}^T\bm{s}$, where 
$\bm{r}=[0 \dots 0 \frac{\partial^p \chi_{ref,p}(\xi)}{\partial \xi^p}]$ and
$\bm{s}=[0 \dots 0 \frac{\partial^p \chi_{ref,p}(\xi)}{\partial \xi^p}\sum_{i=1}^{N_{vp}}W(\xi_{v,i}^r)]$, with $\chi_{ref,p}(\xi)= \sqrt{\frac{2p+1}{2}}\frac{(2p)!}{2^p(p!)^2}\xi^p+\dots+c_0=\sqrt{\frac{2p+1}{2}}c_p\xi^p+\dots +c_0$ is the $p^{\text{th}}$ order normalized Legendre polynomial. Thus, $\frac{\partial^p \chi_{ref,p}(\xi)}{\partial \xi^p} =\sqrt{\frac{2p+1}{2}}c_p p! $. Utilizing that the mass matrix of a normalized Legendre reference basis is an identity matrix, and the Sherman-Morrison formula,

\begin{equation}
\begin{split}
    (\bm{M}+\bm{K} )^{-1} = \bm{T}^{-1}\Big( \bm{I} - \frac{1}{1+c\bm{s}\bm{r}^T}\bm{K}_{ref}\Big)\bm{T}^{-T}\\
    = \bm{M}^{-1} - \frac{1}{1+c(2p+1)(p!c_p)^2}\bm{M}^{-1}\bm{K} \bm{M}^{-1}.
    \end{split}
\end{equation}

\noindent Including Jacobian dependence results in,

\begin{equation}
    (\bm{M}_m+\bm{K}_m)^{-1} 
    = \bm{M}_m^{-1} - \frac{1}{1+c(2p+1)(p!c_p)^2}\bm{M}_m^{-1}\bm{K}_m \bm{M}_m^{-1}.
\end{equation}

\end{proof}

Note that $\bm{K}_m \bm{M}_m^{-1}\bm{S}_\xi = 0$ for linear grids, provided that the flux is not projected to a higher order basis~\cite[Appendix A]{zwanenburg_equivalence_2016}. If the flux is not projected to a higher order basis, we can fully express Eq.~(\ref{eq:split strong}) as,

\begin{equation}
    \begin{split}
    \frac{d}{dt}\hat{\bm{u}}_m(t)^T
   = -\alpha \bm{M}_m^{-1}\bm{S}_{\xi}\hat{\bm{f}}_m^r(t)^T 
   -(1-\alpha)\bm{M}_m^{-1}\bm{\chi}(\bm{\xi}_{v}^r)^T\bm{U}\bm{W}\frac{\partial\bm{\chi}(\bm{\xi}_{v}^r)}{\partial {\xi}}\hat{\bm{u}}_m(t)^T\\
   + \frac{(1-\alpha)}{1+c(2p+1)(p!c_p)^2}\bm{M}_m^{-1}\bm{K}_m\bm{M}_m^{-1} \bm{\chi}(\bm{\xi}_v^r)^T\bm{U}\bm{W}
  \frac{\partial\bm{\chi}(\bm{\xi}_{v}^r)}{\partial {\xi}} \hat{\bm{u}}_m(t)^T\\
	-(\bm{M}_m+\bm{K}_m) ^{-1}\sum_{f=1}^{N_f}
	\bm{\chi}(\bm{\xi}_{f}^r)^T \bm{W}_{f}\diag(\hat{\bm{n}}^r) \Big(  \bm{ {f}}_m^{{*,r}^T }-\alpha\bm{\chi}(\bm{\xi}_{f}^r)\hat{\bm{f}}_m^{{r}}(t)^T-(1-\alpha) { \bm{f}}_{f}^{{r}^T}\Big) .
	\end{split}
\end{equation}

\noindent We numerically demonstrate that this additional term is design order by obtaining the correct orders in Section~\ref{sec: num results}. 



\subsubsection{Discrete Conservation}\label{sec:conservation}

Following the formulation presented in~\cite{gassner2013skew,ranocha2016summation} for conservation, we demonstrate here both local and global conservation using a quadrature of exact integration of at least $2p-1$,

\begin{equation}
\begin{split}
   \text{Continuous: } \int_{\bm{\Omega}_r}  (\bm{\chi}( {\xi^r})\hat{\bm{1}}^T)^T  {J}_m^\Omega (\bm{\chi}( {\xi}^r)\frac{d}{dt}\hat{\bm{u}}_m(t)^T)
    + c\Big(\frac{\partial^p}{\partial \xi^p}(\bm{\chi}( {\xi^r})\hat{\bm{1}}^T)
\Big)^T {J}_m^\Omega\Big(  \frac{\partial^p}{\partial \xi^p}(\bm{\chi}( {\xi^r})\frac{d}{dt}\hat{\bm{u}}_m(t)^T )
\Big)
    d\bm{\Omega}_r\\
 \implies \text{Discrete: }   \hat{\bm{1}}(\bm{M}_m+\bm{K}_m)\frac{d}{dt}\hat{\bm{u}}_m(t)^T,
    \end{split}
\end{equation}

\noindent where $\bm{1}=[1,\dots,1]=\Big(\bm{\chi}(\bm{\xi}_v^r)\hat{\bm{1}}^T\Big)^T$. First, we substitute Eq.~(\ref{eq:split strong}) to show local and global conservation for the proposed split strong form.

\begin{equation}
\begin{split}
   \hat{\bm{1}}(\bm{M}_m+\bm{K}_m) \frac{d}{dt}\hat{\bm{u}}_m(t)^T=
   -\alpha \hat{\bm{1}} \bm{S}_{\xi} \hat{\bm{f}}_m^r(t)^T
   -(1-\alpha)\hat{\bm{1}}\bm{\chi}(\bm{\xi}_v^r)^T\bm{U}\bm{W}\frac{\partial\bm{\chi}(\bm{\xi}_{v}^r)}{\partial {\xi}}\hat{\bm{u}}_m(t)^T\\
	-\hat{\bm{1}}\sum_{f=1}^{N_f}
	\bm{\chi}(\bm{\xi}_{f}^r)^T \bm{W}_{f}\diag(\hat{\bm{n}}^r) \Big(  \bm{{f}}_m^{{*,r}^T}-\alpha\bm{\chi}(\bm{\xi}_{f}^r)\hat{\bm{f}}_m^{{r}}(t)^T-(1-\alpha){\bm{f}}_{f}^{{r}^T}\Big) .
	\end{split}
\end{equation}

\noindent Discretely integrating the first term by parts yields,

\begin{equation}
\begin{split}
    \hat{\bm{1}}(\bm{M}_m+\bm{K}_m)\frac{d}{dt}\hat{\bm{u}}_m(t)^T=
    \alpha \hat{\bm{1}} \bm{S}_{\xi}^T \hat{\bm{f}}_m^r(t)^T
   -(1-\alpha)\hat{\bm{1}}\bm{\chi}(\bm{\xi}_v^r)^T\bm{U}\bm{W}\frac{\partial\bm{\chi}(\bm{\xi}_{v}^r)}{\partial {\xi}}\hat{\bm{u}}_m(t)^T\\
	-\hat{\bm{1}}\sum_{f=1}^{N_f}
	\bm{\chi}(\bm{\xi}_{f}^r)^T \bm{W}_{f}\diag(\hat{\bm{n}}^r) \Big(  \bm{{f}}_m^{{*,r}^T}-(1-\alpha){\bm{f}}_{f}^{{r}^T}\Big) .
	\end{split}\label{eq:strong loc cons deriv}
\end{equation}

\noindent Note that for the first volume term,

\begin{equation}
\begin{split}
    \hat{\bm{1}} \bm{S}_{\xi}^T=\Big(\bm{S}_{\xi}\hat{\bm{1}}^T \Big)^T 
    =\Big( \bm{\chi}(\bm{\xi}_{v}^r)^T\bm{W}\frac{\partial\bm{\chi}(\bm{\xi}_{v}^r)}{\partial  {\xi}} \hat{\bm{1}}^T \Big)^T = 
    \Big( \bm{\chi}(\bm{\xi}_{v}^r)^T\bm{W}\frac{\partial}{\partial  {\xi}} {\bm{1}}^T \Big)^T=\bm{0},
    \end{split}
\end{equation}

\noindent is the derivative of a constant, and hence eliminated. Using $\bm{W}\frac{\partial\bm{\chi}(\bm{\xi}_{v}^r)}{\partial {\xi}} = \bm{\Pi}_m^{{p}^T}\bm{S}_\xi$ 
on the second volume term and discretely integrating by parts we obtain,

\begin{equation}
\begin{split}
    \hat{\bm{1}}(\bm{M}_m+\bm{K}_m)\frac{d}{dt}\hat{\bm{u}}_m(t)^T=(1-\alpha)\hat{\bm{1}}\bm{\chi}(\bm{\xi}_{v}^r)^T\bm{U}\bm{\Pi}_m^{{p}^T}\bm{S}_{\xi}^T\hat{\bm{u}}_m(t)^T\\
    -(1-\alpha)\hat{\bm{1}}	\bm{\chi}(\bm{\xi}_{v}^r)^T\bm{U}\bm{\Pi}_m^{{p}^T}
    \sum_{f=1}^{N_f}
    \bm{\chi}( \bm{\xi}_{f}^r)^T \bm{W}_{f}\diag(\hat{\bm{n}}^r) \bm{\chi}( \bm{\xi}_{f}^r)\hat{\bm{u}}_m(t)^T\\
	-\hat{\bm{1}}\sum_{f=1}^{N_f}
	\bm{\chi}(\bm{\xi}_{f}^r)^T \bm{W}_{f}\diag(\hat{\bm{n}}^r) \Big(  \bm{{f}}_m^{{*,r}^T}-(1-\alpha){\bm{f}}_{f}^{{r}^T}\Big).
	\end{split}\label{eq:local cons 2}
\end{equation}

\noindent Adding a half of Eq.~(\ref{eq:strong loc cons deriv}) and~(\ref{eq:local cons 2}), and noticing that

\begin{equation}
\begin{split}
\hat{\bm{1}}\bm{\chi}(\bm{\xi}_{v}^r)^T\bm{U}\bm{\Pi}_m^{{p}^T}\bm{S}_{\xi}^T\hat{\bm{u}}_m(t)^T
=(\bm{\chi}(\bm{\xi}_{v}^r)\hat{\bm{1}}^T)^T\bm{U}\bm{\Pi}_m^{{p}^T}\bm{S}_{\xi}^T\hat{\bm{u}}_m(t)^T\\
= (\bm{\Pi}_m^{{p}}\bm{u}_m^T)^T\bm{S}_{\xi}^T\hat{\bm{u}}_m(t)^T
=\hat{\bm{u}}_m(t)\bm{S}_{\xi}^T\hat{\bm{u}}_m(t)^T
=\Big(\hat{\bm{u}}_m(t)\bm{S}_{\xi}\hat{\bm{u}}_m(t)^T\Big)^T
\end{split}
\end{equation}

\noindent is a scalar, we can drop the transpose and the volume terms cancel and hence,

\begin{equation}
    \begin{split}
    \hat{\bm{1}}(\bm{M}_m+\bm{K}_m)\frac{d}{dt}\hat{\bm{u}}_m(t)^T=
    -\frac{(1-\alpha)}{2}\sum_{f=1}^{N_f}\hat{\bm{u}}_m(t)\bm{\chi}(\bm{\xi}_{f}^r)^T \bm{W}_{f}\diag(\hat{\bm{n}}^r) \bm{\chi}(\bm{\xi}_{f}^r)\hat{\bm{u}}_m(t)^T \\
	-\hat{\bm{1}}\sum_{f=1}^{N_f}
	\bm{\chi}(\bm{\xi}_{f}^r)^T \bm{W}_{f}\diag(\hat{\bm{n}}^r) \Big(  \bm{{f}}_m^{{*,r}^T}-(1-\alpha){\bm{f}}_{f}^{{r}^T}\Big).
	\end{split}\label{eq: temp for now}
\end{equation}

Finally, if we consider the first term in Eq.~(\ref{eq: temp for now}) evaluated at a single facet cubature node $k$,

\begin{equation}
\begin{split}
    \frac{1}{2}\hat{\bm{u}}_m(t) \bm{\chi}({\xi}_{f,k}^r)^T {W}_{f,k}\hat{{n}}^r_{f,k} \bm{\chi}({\xi}_{f,k}^r)\hat{\bm{u}}_m(t)^T =
    \frac{1}{2}\Big(u_{f,k} W_{f,k}\hat{n}_{f,k}^r u_{f,k} \Big)\\
    = 1 W_{f,k} \hat{n}_{f,k}^r (\frac{1}{2} u_{f,k}^2)
    =\hat{\bm{1}}\bm{\chi}({\xi_{f,k}^r})^T {W}_{f,k}\hat{{n}}^r_{f,k} {{{f}}_{f,k}^{{r}}},
    \end{split}
\end{equation}

\noindent then we are left with,

\begin{equation}
    \hat{\bm{1}}(\bm{M}_m+\bm{K}_m)\frac{d}{dt}\hat{\bm{u}}_m(t)^T=-\hat{\bm{1}}\sum_{f=1}^{N_f}
	\bm{\chi}(\bm{\xi}_{f}^r)^T \bm{W}_{f}\diag(\hat{\bm{n}}^r)   \bm{{f}}_m^{{*,r}^T},\label{eq:strong cons}
\end{equation}

\noindent which concludes the proof for the strong ESFR split form's local and global conservation. 

\subsubsection{Discrete Energy Stability}\label{sec:stability proof}

We consider the broken Sobolev-norm in Eq.~(\ref{eq:init energy ESFR}) to demonstrate nonlinear stability. Analyzing the uncollocated ESFR split strong form, we insert Eq.~(\ref{eq:split strong}) into the energy balance, and notice that $(\bm{M}_m+\bm{K}_m)(\bm{M}_m+\bm{K}_m)^{-1}=\bm{I}$, the identity matrix. That is, when incorporating the ESFR filter on the nonlinear volume integral, the dense ESFR norm cancels off, and we obtain,

\begin{equation}
\begin{split}
    \frac{1}{2}\frac{d}{dt}\|\bm{u}\|_{M_m+K_m}^2=
    -\alpha\hat{\bm{u}}_m(t)\bm{S}_{\xi}\hat{\bm{f}}_m^r(t)^T
    -(1-\alpha)\hat{\bm{u}}_m(t)\bm{\chi}(\bm{\xi}_{v}^r)^T\bm{U}\bm{W}
    \frac{\partial \bm{\chi}(\bm{\xi}_{v}^r)}{\partial  {\xi}}\hat{\bm{u}}_m(t)^T\\
    -\hat{\bm{u}}_m(t)\sum_{f=1}^{N_f}
	\bm{\chi}(\bm{\xi}_{f}^r)^T \bm{W}_{f}\diag(\hat{\bm{n}}^r) \Big( {\bm{f}}_m^{{*,r}^T}-\alpha	\bm{\chi}(\bm{\xi}_{f}^r)\hat{\bm{f}}_m^{{r}}(t)^T-(1-\alpha){\bm{f}}_{f}^{{r}^T} \Big).
	\end{split}\label{eq:init strong energy}
\end{equation}

\noindent Using discrete integration by parts, Eq.~(\ref{eq: integration by parts}), on the first volume term results in,

\begin{equation}
\begin{split}
   \frac{1}{2} \frac{d}{dt}\|\bm{u}\|_{M_m+K_m}^2=
    \alpha\hat{\bm{u}}_m(t)\bm{S}_{\xi}^T\hat{\bm{f}}_m^r(t)^T
    -(1-\alpha)\hat{\bm{u}}_m(t)\bm{\chi}(\bm{\xi}_v^r)^T\bm{U}\bm{W}
    \frac{\partial \bm{\chi}(\bm{\xi}_v^r)}{\partial  {\xi}}\hat{\bm{u}}_m(t)^T\\
    -\hat{\bm{u}}_m(t)\sum_{f=1}^{N_f}
	\bm{\chi}(\bm{\xi}_{f}^r)^T \bm{W}_{f}\diag(\hat{\bm{n}}^r) ({\bm{f}}_m^{{*,r}^T}-(1-\alpha) { \bm{f}}_f^{{r}^T}).
    \end{split}
\end{equation}


\noindent Note that $\hat{\bm{u}}_m(t)\bm{S}_{\xi}^T\hat{\bm{f}}_m^r(t)^T = \hat{\bm{f}}_m^r(t)\bm{S}_\xi\hat{\bm{u}}_m(t)^T = \Big(\bm{\chi}(\bm{\xi}_v^r) \hat{\bm{u}}_m(t)^T\Big)^T\bm{U}\bm{\Pi}_m^{{p}^T}\bm{S}_{\xi}\hat{\bm{u}}_m(t)^T$ is a scalar and thus,

\begin{equation}
\begin{split}
   \frac{1}{2} \frac{d}{dt}\|\bm{u}\|_{M_m+K_m}^2=
    \frac{1}{2}\alpha\Big(\bm{\chi}(\bm{\xi}_v^r) \hat{\bm{u}}_m(t)^T\Big)^T\bm{U}\bm{\Pi}_m^{{p}^T}\bm{S}_{\xi}\hat{\bm{u}}_m(t)^T\\
    -(1-\alpha)\hat{\bm{u}}_m(t)\bm{\chi}(\bm{\xi}_v^r)^T\bm{U}\bm{\Pi}_m^{{p}^T}\bm{S}_{\xi}\hat{\bm{u}}_m(t)^T\\
    -\hat{\bm{u}}_m(t)\sum_{f=1}^{N_f}
	\bm{\chi}(\bm{\xi}_{f}^r)^T \bm{W}_{f}\diag(\hat{\bm{n}}^r) ({\bm{f}}_m^{{*,r}^T}-(1-\alpha){\bm{f}}_f^{{r}^T}).
    \end{split}
\end{equation}

\noindent Choosing $\alpha=\frac{2}{3}$ has the volume terms vanish, and we are left with,

\begin{equation}
   \frac{1}{2} \frac{d}{dt}\|\bm{u}\|_{M_m+K_m}^2=
    -\hat{\bm{u}}_m(t)\sum_{f=1}^{N_f}
	\bm{\chi}(\bm{\xi}_{f}^r)^T \bm{W}_{f}\diag(\hat{\bm{n}}^r)( { \bm{f}}_m^{{*,r}^T}-\frac{1}{3}{\bm{f}}_f^{{r}^T}).\label{eq:energy strong}
\end{equation}

Since our energy balance Eq.~(\ref{eq:energy strong}) only incorporates terms evaluated on the surface, we reuse observations from Gassner~\cite{gassner2013skew}. After considering an edge term, assuming all interior (left) cells' outward pointing normal is 1, letting $w_0$ represent the solution at the right of the edge, and $v_p$ represent the solution at the left of the face, the surface contribution for Eq.~(\ref{eq:energy strong}) is,

\begin{equation}
    \text{surface contribution} = (w_0-v_p)^2\Big(\frac{1}{12}(w_0-v_p) -\lambda \Big),
\end{equation}

\noindent where the following numerical flux is used,

\begin{equation}
    { {f}}_m^{{*,r}}=\frac{1}{2}\Big(\frac{w_0^2}{2}+\frac{v_p^2}{2} \Big) - \lambda(w_0-v_p).\label{eq:num flux}
\end{equation}

\noindent This directly leads to the stability criterion of,

\begin{equation}
    \lambda \geq \frac{1}{12}(w_0-v_p).\label{eq: stab lamba}
\end{equation}

\noindent The local Lax-Freidrichs numerical flux ensures energy and entropy stability, since,

\begin{equation}
    \lambda_{LLF}=\frac{1}{2}\text{max}(|w_0|,|v_p|)\geq \frac{1}{12}(w_0-v_p).\label{eq: LLF lamba}
\end{equation}

%% file: body/Results.tex
\section{Numerical Results}\label{sec: num results}

In this section, we use the open-source Parallel High-order Library for PDEs (\texttt{PHiLiP})~\cite{shi2021full} and consider similar test cases as used in~\cite{gassner2013skew} and~\cite{ranocha2016summation},

\begin{equation}
\begin{split}
    \frac{\partial}{\partial t}u+\frac{\partial}{\partial x}(\frac{u^2}{2})=q(x,t),\: x\in[0,2],
\end{split}
\end{equation}

\noindent with periodic boundary conditions. Using values of $c_{DG}$, $c_+$, and $c_{10^4}=10^4$, we first demonstrate that our proposed strong ESFR split form satisfies the energy/entropy stability criteria derived in Sec.~\ref{sec:split forms} on collocated GLL nodes, uncollocated GL nodes, and uncollocated GL nodes with overintegration; with $q(x,t)=0$ and $u(x,0)=\sin(\pi x)+0.01$. Our basis functions are Lagrange polynomials constructed on GLL nodes. Next, we verify that we observe the correct orders of accuracy for the strong ESFR split form by using $q(x,t)=\pi\sin{(\pi(x-t))}(1-\cos{(\pi(x-t))})$, $u(x,0)=\cos(\pi x)$ and $u_{\text{exact}}(x,t)=\cos{(\pi(x-t))}$. We utilize the generality of the scheme by testing both uncollocated $N_{vp}=p+1$ and an uncollocated overintegration scheme on $N_{vp}=p+3$; both on GL nodes.


Since our proposed splitting includes the additional filtering of $(\bm{M}_m+\bm{K}_m)^{-1}$ on the second volume term, which does not occur in the classical ESFR scheme, we verify that it is needed for stability and that it does not effect the order of accuracy. For readability, \enquote{Cons. DG} refers to the conservative DG scheme Eq.~(\ref{eq:dg strong}), \enquote{ESFR Split} refers to our proposed splitting in Eq.~(\ref{eq:split strong}), and \enquote{Classical ESFR Split} refers to the strong split form with a classical ESFR implementation, where $(\bm{M}_m+\bm{K}_m)^{-1}$ is only applied on the face terms. All schemes were conservative on the order of $1\times 10^{-16}$.

\subsection{Energy Verification}

For the collocated results, we integrate the solution and fluxes on Gauss-Lobatto-Legendre quadrature nodes. For the uncollocated results, we integrate the solution and fluxes on Gauss-Legendre quadrature nodes. The solution is integrated in time using RK4 with a timestep of $\Delta t =1e-4$ until a final time of $t_f=3$ s, and the grid is partitioned into $M=8$ uniform elements. \enquote{ECON} refers to the energy conserving numerical flux (the equality in Eq.~(\ref{eq: stab lamba})) and \enquote{LF} refers to the Lax-Friedrichs numerical flux, using the value from Eq.~(\ref{eq: LLF lamba}). Tables~\ref{tab:coll}, ~\ref{tab: uncoll}, and~\ref{tab: overint} present the energy results. 

 \begin{table}[t]
\resizebox{\textwidth}{!}{
 \begin{tabular}{c c c c }
Scheme &  Flux & Energy Conserved $\mathcal{O}$(1e-12) & Energy Monotonically Decrease \\ \hline 
Cons. DG & ECON & No & No \\ \hline 
Cons. DG & LF & No & No \\ \hline 
EFSR Split $c_{DG}$ \footnotemark[1]& ECON& \textbf{Yes} & \textbf{Yes} \\\hline 
EFSR Split $c_{DG}$ \footnotemark[1]& LF & No & \textbf{Yes} \\ \hline 
EFSR Split $c_{+}$ & ECON  & \textbf{Yes} & \textbf{Yes} \\ \hline 
EFSR Split $c_{+}$ & LF & No &\textbf{Yes} \\ \hline 
EFSR Split $c_{10^4}$ & ECON & \textbf{Yes} & \textbf{Yes} \\ \hline 
EFSR Split $c_{10^4}$ & LF & No & \textbf{Yes}\\ \hline 
EFSR Classical Split $c_{+}$ & ECON  & No & No \\ \hline 
EFSR Classical Split $c_{+}$ & LF & No & No \\  \hline 
EFSR Classical Split $c_{HU}$ \footnotemark[1]& ECON  & No & No \\ \hline 
EFSR Classical Split $c_{HU}$ \footnotemark[1] & LF & No & No \\  \hline 
ESFR Classical Split $c_{HU}$ Lumped-Lobatto\footnotemark[1] & ECON & \textbf{Yes}&\textbf{Yes}\\\hline 
ESFR Classical Split $c_{HU}$ Lumped-Lobatto \footnotemark[1]& LF &No & \textbf{Yes}\\
\end{tabular} 
 }
 \caption{Energy Results $p=4$,$5$ Collocated Schemes $N_{vp}=p+1$} \label{tab:coll}
 \end{table} 
 
 \footnotetext[1]{Note that the $\text{g}_2$ lumped-Lobatto scheme presented by Huynh~\cite{huynh_flux_2007}, and used by Abe \textit{et al}.~\cite{abe2018stable}, is equivalent to a collocated DG scheme on GLL nodes and not equivalent to using $c_{HU}$ with a collocated $\bm{K}_m$ operator~\cite{de2014connections}.}
 


 \begin{table}[t]
\resizebox{\textwidth}{!}{
 \begin{tabular}{c c c c }
Scheme &  Flux & Energy Conserved $\mathcal{O}$(1e-12) & Energy Monotonically Decrease \\ \hline 
Cons. DG & ECON & No & No \\ \hline 
Cons. DG & LF & No & No \\ \hline 
EFSR Split $c_{DG}$ & ECON & \textbf{Yes} & \textbf{Yes} \\ \hline 
EFSR Split $c_{DG}$ & LF & No & \textbf{Yes} \\ \hline 
EFSR Split $c_{+}$ & ECON  & \textbf{Yes} & \textbf{Yes} \\ \hline 
EFSR Split $c_{+}$ & LF & No & \textbf{Yes} \\ \hline 
EFSR Split $c_{10^4}$ & ECON & \textbf{Yes} & \textbf{Yes} \\ \hline 
EFSR Split $c_{10^4}$ & LF & No &\textbf{Yes} \\ \hline 
EFSR Classical Split $c_{+}$ & ECON  & No & No \\ \hline 
EFSR Classical Split $c_{+}$ & LF & No & No \\\hline 
EFSR Classical Split $c_{HU}$ \footnotemark[1]& ECON  & No & No \\ \hline 
EFSR Classical Split $c_{HU}$ \footnotemark[1] & LF & No & No \\  \hline 
ESFR Classical Split $c_{HU}$ Lumped-Lobatto \footnotemark[2]& ECON & N/A & N/A\\\hline 
ESFR Classical Split $c_{HU}$ Lumped-Lobatto \footnotemark[2]& LF &N/A & N/A\\
\end{tabular} 
 }
 \caption{Energy Results $p=4$,$5$ Uncollocated Schemes $N_{vp}=p+1$}  \label{tab: uncoll}
 \end{table}

  \footnotetext[2]{N/A refers to \enquote{Not Available}. The $\text{g}_2$ lumped-Lobatto scheme cannot be run on uncollocated volume nodes since it would not make sense lumping the mode on the boundary.}

  \begin{table}[hbt!]
\resizebox{\textwidth}{!}{
 \begin{tabular}{c c c c }
Scheme &  Flux & Energy Conserved $\mathcal{O}$(1e-12) & Energy Monotonically Decrease \\ \hline 
Cons. DG & ECON & No & No \\ \hline 
Cons. DG & LF & No & No \\ \hline 
EFSR Split $c_{DG}$ & ECON & \textbf{Yes} & \textbf{Yes} \\ \hline 
EFSR Split $c_{DG}$ & LF & No & \textbf{Yes} \\ \hline 
EFSR Split $c_{+}$ & ECON  & \textbf{Yes} & \textbf{Yes} \\ \hline 
EFSR Split $c_{+}$ & LF & No & \textbf{Yes} \\ \hline 
EFSR Split $c_{10^4}$ & ECON & \textbf{Yes}& \textbf{Yes} \\ \hline 
EFSR Split $c_{10^4}$ & LF & No & \textbf{Yes} \\ \hline 
EFSR Classical Split $c_{+}$ & ECON  & No & No \\ \hline 
EFSR Classical Split $c_{+}$ & LF & No & No \\\hline 
EFSR Classical Split $c_{HU}$ \footnotemark[1]& ECON  & No & No \\ \hline 
EFSR Classical Split $c_{HU}$ \footnotemark[1] & LF & No & No \\  \hline 
ESFR Classical Split $c_{HU}$ Lumped-Lobatto \footnotemark[2]& ECON & N/A & N/A\\\hline 
ESFR Classical Split $c_{HU}$ Lumped-Lobatto \footnotemark[2]& LF &N/A & N/A\\
\end{tabular} 
 }
 \caption{Energy Results $p=4$,$5$ Uncollocated Schemes Overintegrated $N_{vp}=p+3$}  \label{tab: overint}
 \end{table}


\subsection{Orders of Accuracy (OOA)}

To compute the $\text{L}_2$-error, an overintegration of $p+10$ was used in calculating the error to provide a sufficient strength for the purpose of accuracy. 

\begin{equation}
    \text{L}_2 -\text{error} = \sqrt{\sum_{m=1}^{M}{\int_\Omega{(u_m-u)^2d\Omega}}}
    =\sqrt{\sum_{m=1}^{M}(\mathbf{u}_m^T-\mathbf{u}_{exact}^T)\mathbf{W}\mathbf{J}_m(\mathbf{u}_m-\mathbf{u}_{exact})}.
\end{equation}

The $\text{L}_2$-errors are shown for the test case described above for Cons. DG and strong ESFR Split with $c_{DG}$ to provide a direct comparison of the influence of splitting the volume and face terms on the OOA. Also, strong ESFR Split with $c_+$, and strong ESFR Classical Split with $c_+$ are tested to give a direct comparison of the influence of $(\bm{M}_m+\bm{K}_m)^{-1}$ being applied on the non-conservative volume term for accuracy. We use a $t_f=1$ s, $\Delta t =1\times 10^{-4}$, and a Lax-Friedrichs numerical flux for the OOA test. We demonstrate the OOA for $p=4$ uncollocated $N_{vp}=p+1$ in table~\ref{tab:uncoll p4} and uncollocated overintegration in table~\ref{tab:overint p4}; similarly for $p=5$ in tables~\ref{tab:uncoll p5} and~\ref{tab:overint p5}. All schemes yield the expected convergence rates of $p+1$.

 \begin{table}[H]
\resizebox{\textwidth}{!}{
 \begin{tabular}{c c c c c c c c c}
dx & Cons. DG & OOA &ESFR Split $c_{DG}$ & OOA
& ESFR Split $c_{+}$ & OOA& ESFR Classical Split $c_+$ & OOA \\ \hline 

 2.50e-02&   7.82e-06 & -  
 & 7.72e-06 & -  & 2.22e-04  & - &  1.42e-04 & -\\
1.25e-02 &  1.94e-07  &  5.33 
 & 1.93e-07   &  5.32 & 6.77e-06  &  5.04 & 4.18e-06 & 5.09 \\
 6.25e-03&  5.17e-09  &  5.23 
 &  5.17e-09  & 5.23 &  1.98e-07 & 5.10 &  1.28e-07 &  5.03\\
 3.13e-03&   1.48e-10  &  5.12 
 &  1.48e-10 & 5.12  &  6.30e-09 &  4.97 &  4.21e-09 &  4.93\\
 1.56e-03  &4.55e-12 & 5.02 
 &   4.55e-12  & 5.02   &  1.96e-10  &   5.00    &   1.33e-10  & 4.98

\end{tabular} 
 }
 \caption{Convergence Table $p=4$ $N_{vp}=p+1$} \label{tab:uncoll p4}
 \end{table}

\begin{table}[H]
\resizebox{\textwidth}{!}{
 \begin{tabular}{c c c c c c c c c}
dx & Cons. DG & OOA &ESFR Split $c_{DG}$ & OOA 
& ESFR Split $c_{+}$ & OOA& ESFR Classical Split $c_+$ & OOA \\ \hline 

2.08e-02 &  1.65e-07  & - & 1.57e-07 & - 
 & 2.25e-05  & - &   1.24e-05 & -\\
 1.04e-02& 2.31e-09  & 6.15 & 2.31e-09 &  6.09
 &  4.24e-07 & 5.73 & 2.35e-07 & 5.72 \\
 5.21e-03&  3.55e-11 & 6.02 &3.56e-11  &  6.02
& 8.00e-09 & 5.73  &  4.84e-09 & 5.60 \\
 2.60e-03&  - & - & -  & - 
& 1.54e-10  &  5.70 & 9.99e-11  & 5.60 \\
  1.30e-03& - & - & - &  -
 & 2.84e-12 & 5.76 &  1.88e-12&  5.73

\end{tabular} 
 }
 \caption{Convergence Table $p=5$ $N_{vp}=p+1$} \label{tab:uncoll p5}
 \end{table}

 
 
 
 \begin{table}[H]
\resizebox{\textwidth}{!}{
 \begin{tabular}{c c c c c c c c c}
dx & Cons. DG & OOA &ESFR Split $c_{DG}$ & OOA
& ESFR Split $c_{+}$ & OOA& ESFR Classical Split $c_+$ & OOA \\ \hline 

 2.50e-02&  7.37e-06   & -  
 &   7.37e-06 & -  &  2.21e-04   & - &  9.32e-05  & -\\
1.25e-02 &  1.91e-07  &   5.27
 &    1.91e-07 &  5.27 &  6.76e-06  & 5.03   &  1.16e-06 &  6.33 \\
 6.25e-03&   5.15e-09  &  5.21 
 &  5.15e-09  & 5.21 &  1.97e-07 &  5.10 &  1.94e-08  &  5.90 \\
 3.13e-03&   1.48e-10   &  5.12 
 &  1.48e-10 & 5.12  &  6.30e-09 & 4.97  &  5.05e-10 &   5.26 \\
1.56e-03 &  4.55e-12    &   5.02
 &  4.55e-12 & 5.02  &  1.96e-10 &  5.00 & 1.52e-11 & 5.06

\end{tabular} 
 }
 \caption{Convergence Table $p=4$ Overintegrated $N_{vp}=p+3$} \label{tab:overint p4}
 \end{table}

\begin{table}[H]
\resizebox{\textwidth}{!}{
 \begin{tabular}{c c c c c c c c c}
dx & Cons. DG & OOA &ESFR Split $c_{DG}$ & OOA 
& ESFR Split $c_{+}$ & OOA& ESFR Classical Split $c_+$ & OOA \\ \hline 

2.08e-02 &   1.56e-07  & - & 1.56e-07& - 
 &  2.24e-05  & - &   1.46e-06 & -\\
 1.04e-02&  2.33e-09  & 6.07  &  2.33e-09 &  6.07 
 &  4.23e-07 & 5.72  &   1.01e-08&  7.18\\
 5.21e-03&  3.57e-11  & 6.03 & 3.57e-11 &   6.03
& 8.00e-09 & 5.72  &  1.09e-10 & 6.53 \\
 2.60e-03&  - & - & -  & - 
&  1.54e-10  &   5.70 &  1.66e-12 &6.04 \\
  1.30e-03& - & - & - &  -
 & 2.84e-12 & 5.76& - &   -

\end{tabular} 
 }
 \caption{Convergence Table $p=5$ Overintegrated $N_{vp}=p+3$} 
 \label{tab:overint p5}
 \end{table}

%% file: body/Conclusion.tex
\section{Conclusion}

This paper derived dense, modal or nodal, ESFR schemes in both strong and weak forms that resulted in provable nonlinear stability and conservation
by incorporating the ESFR filter operator on both the volume and surface terms. It was shown that by considering split forms with respect to the stiffness operator rather than the differential operator, discrete integration by parts was embedded in the discretization. The stability criteria, conservation, and convergence orders were numerically verified for a wide range of general ESFR schemes.

%% file: body/Acknowledgements.tex
\section{Acknowledgements}

We would like to gratefully acknowledge the financial support of the Natural Sciences and Engineering Research Council of Canada Discovery Grant Program and McGill University.